\documentclass[12pt]{article}
\usepackage{latexsym,amsfonts,amsthm,amsmath,amscd,amssymb}
\usepackage{tikz-cd}

\newtheorem{lemma}{Lemma}[section]
\newtheorem{thm}[lemma]{Theorem}
\newtheorem{rem}[lemma]{Remark}

\newtheorem{cor}[lemma]{Corollary}

\newtheorem{example}[lemma]{Example}

\begin{document}

\title{Finite-dimensional diffeological vector spaces being and not being coproducts}

\author{Ekaterina~{\textsc Pervova}}

\maketitle

\begin{abstract}
\noindent It is known that for a topological vector space it is possible to be the coproduct of two of its subspaces in the category of vector spaces but while not being the coproduct of the same subspaces in the category of topological vector spaces. There are however wide classes of spaces where this cannot occur, notably finite-dimensional spaces (but also some infinite-dimensional ones, for instance, Banach spaces). In contrast, this kind of phenomen occurs easily (and frequently, as we here show) for finite-dimensional diffeological vector spaces, where its numerous instances are readily obtained in any dimension starting from $2$. After briefly reviewing what is known on this question in some classical categories, we provide an overview of this phenomenon and some of its implications for finite-dimensional diffeological vector spaces, indicating briefly its connections with some other subjects.
\vspace{5mm}


\end{abstract}

\section{Vector spaces in general as coproducts or not}

There is a simple mathematical curiosity that arises when a vector space is endowed with an extra structure, such as that of a topological vector space, a smooth structure of sorts, and so on. In the most general terms, this curiosity can be described as follows: let $C$ be a category endowed with a faithful functor $\Xi$ to the category \textbf{Vect} of vector spaces, and let $c,a,b$ be objects of $C$ such that $\Xi(c)$ is the coproduct in the category \textbf{Vect} of $\Xi(a)$ and $\Xi(b)$; then one may wonder whether $c$ itself is the coproduct of $a$ and $b$ in the initial category $C$. 

\paragraph{Topological vector spaces} The most obvious instance of course is that of the category $C$ being the category of topological vector spaces and $\Xi$ being the forgetful functor into \textbf{Vect}; given a topological vector space $V$ whose underlying vector space decomposes into a direct sum of two of its subspaces, $\Xi(V)=\Xi(V_1)\oplus\Xi(V_2)$, where $V_1$ and $V_2$ are vector subspaces of $V$ endowed with the subspace topology, one may wonder whether the product topology on $V_1\oplus V_2$ (relative to their subspace topologies inherited from $V$) coincides with the topology of $V$ itself (this is usually expressed as $V$ being an \emph{algebraic direct sum} of $V_1$ and $V_2$ vs. it being their \emph{topological direct sum}). 

Now, in this case it is rather well-known that the answer is positive for any finite-dimensional $V$, while it may turn out to be negative in the case of infinite dimension. Indeed, any finite-dimensional Hausdorff topological vector space is homeomorphic to $\mathbb{R}^n$ (if $V$ is over $\mathbb{R}$, to $\mathbb{C}^n$ if it is over $\mathbb{C}$), while in the case of non-Hausdorff topological vector spaces it suffices to recall that every such space is the product of a Hausdorff space and an indiscrete space.

Thus, for topological vector spaces the question of an algebraic direct sum of vector subspaces being also a topological direct sum is a matter limited to infinite dimension. There indeed the two concepts may relatively easily turn out to be distinct as can be illustrated by the space $V$ defined as the subspace of $C[0,1]$ (the space of all continuous functions with the uniform convergence topology) given by $V=\mbox{Span}(M,N)$ with $M$ the space of all polynomials and $N$ the span of some nonpolynomial continuous function, and endowed with the subspace topology relative to its inclusion in $C[0,1]$: indeed, $V=M\oplus N$ algebraically, however its topology is larger than the direct sum topology relative to the inherited topologies on $M$ and $N$.\footnote{It must be said that I encountered this example in an online discussion at \textit{https://math.stackexchange.com/questions/78917/example-of-a-topological-vector-space-such-that-e-m-oplus-n-algebraically}.} On the other hand, for some specific classes of infinite-dimensional vector spaces (such as, for instance, Banach spaces) it may turn out that for every member of that class any its decomposition into an algebraic direct sum is also a decomposition into a topological direct sum (see \cite{banach}).

\paragraph{Vector spaces with a ``smooth'' structure} Similarly to the case of topological vector spaces, if $V$ is a vector space endowed with some type of a smooth structure such as that of a finite-dimensional manifold (relative to which the vector space operations on $V$ are smooth), a Hilbert space, a Fr\'echet space, etc, and the subspaces $V_1$, $V_2$ of $V$ are such that they inherit from $V$ an appropriate smooth structure (as is automatic in many cases), then again one may wonder whether the corresponding product smooth structure on $V_1\oplus V_2$ coincides with that of $V$. Here again, if the dimension of $V$ is finite, one observes that $V$ is in particular a connected Abelian Lie group with respect to the addition operation; all such groups are of form $\mathbb{R}^k\times(\mathbb{S}^1)^l$ (\cite{abelianLie}) and so any such vector space must again be some $\mathbb{R}^n$ with its usual smooth structure. 

\paragraph{Finite dimension appearing} Contrary to the above-listed cases, it can be noted that in categories different from \textbf{Vect} a phenomenon in some ways similar may occur even in finite dimension. For instance, as follows from \cite{freedman82}, \cite{donaldson}, \cite{freedman84}, there exist smooth manifolds, such as the K3 surface, that admit topological decompositions (into a connected sum) but not smooth decompositions. The analogy is not wholly complete, however, since in general a connected sum is not the coproduct of its factors.

\paragraph{Diffeological vector spaces} Now, what happens when $C$ is the category of diffeological vector spaces, is that a finite-dimensional diffeological vector space (and a rather simple one at that) may decompose as an algebraic direct sum of two of its subspaces without being their (diffeological) smooth direct sum\footnote{Let $V$ be a diffeological vector space, and let $V_1$, $V_2$ be two of its vector subspaces such that the vector space underlying $V$ coincides with $V_1\oplus V_2$. The diffeological vector space $V$ is said to decompose as a smmoth sum of $V_1$ and $V_2$ if the diffeology on $V$ coincides with the product diffeology on $V_1\oplus V_2$. relative to the subset diffeologies on $V_1$ and $V_2$ inherited from $V$.} \cite{me2018}. Thus, the just-mentioned simple example indicates that a number of standard trivial properties involving direct sum decompositions of vector spaces cannot be guaranteed in the diffeological context, even for very simple finite-dimensional examples: a given direct sum decomposition may not be smooth (and checking whether one is or is not so may not be wholly trivial, as the example in the next section indicates), a given subspace may not --- \emph{a priori} --- be complemented\footnote{Meaning that it may not be part of any smooth direct sum decimposition of the ambient space.}, and, also \emph{a priori}, a given finite-dimensional diffeological vector space may not admit any smooth decomposition into a direct sum at all. We are not yet aware of specific examples as to the latter possibilities (we just give some indications as to how such examples could be obtained, possibly exploiting a certain conjectural property of so-called non-Baire functions, \emph{i.e.} functions not belonging to any Baire class, \cite{lebesgue}), although it is known (\cite{me2018}) that their consideration should be limited to the vector spaces with trivial diffeological dual: indeed, every finite-dimensional diffeological vector space $V$ admits a (non-unique) smooth decomposition into a direct sum of a subspace diffeomorphic to its diffeological dual (hence whose subset diffeology is standard) and the so-called \emph{maximal isotropic subspace}, which is defined as the intersection of kernels of all smooth linear $\mathbb{R}$-valued functions.

Still another host of issues regards the classic direct sum decomposition of a vector space $V$ relative to a given linear map $f:V\to W$, that as $V=\mbox{Ker}(f)\oplus\mbox{Im}(f)$. Unlike the standard context where the meaning of such equality is immediately understood, in the diffeological context and to our purposes much specification is needed as to what exactly should be meant by the subspace $\mbox{Im}(f)$. Indeed, by the standard definition $\mbox{Im}(f)$ is a subspace of the target space $W$ (and the above equality indicates in fact an isomorphism), and so its most natural diffeology is the subset diffeology relative to the diffeology on $W$. However, this diffeology, while guaranteed to contain the pushforward of the diffeology of $V$ by $f$, may easily turn out to be larger (which is an internal property of $W$). If this is the case, $V$ will certainly not be diffeomorphic to $\mbox{Ker}(f)\oplus\mbox{Im}(f)$ --- but for reasons fundamentally different from direct sums being or not being smooth.

Thus, the most natural way to pose the question appears to be that to ask whether $V$ is always diffeomorphic to $\mbox{Ker}(f)\oplus\mbox{Im}(f)$, where $\mbox{Im}(f)$ is endowed with the pushforward of the diffeology of $V$ by $f$. Still another angle on the same question relates it to the question of smoothness of direct sums. Indeed, $\mbox{Ker}(f)$ admits in general a multitude of algebraic complements in $V$; one may, on one hand, wonder whether any or even all of these complements, considered with the corresponding subset diffeologies, are diffeomorphic to $\mbox{Im}(f)$ with its pushforward diffeology --- and on the other hand whether any of these subspaces complements $\mbox{Ker}(f)$ smoothly, the two questions being \emph{a priori} distinct.

In what follows we first illustrate how even establishing a smoothness of a given direct sum decomposition of a diffeological vector space could be tricky, and then, after recalling the known example of a non-smooth direct sum decomposition, we provide some considerations regarding the existence of non-complemented subspaces, and then that of finite-dimensional diffeological vector spaces that do not admit any smooth decomposition into a direct sum.

\paragraph{Acknowledgment} I would like to thank the organizers of the \emph{Special session on recent advances in diffeology and their applications} for their kind invitation.

\section{When a maximal isotropic subspace is a (non-obvious) coproduct}

In this section we consider in detail the following example.

\begin{example}
Let $V$ be $\mathbb{R}^2$ endowed with the vector space diffeology generated by the following two maps:
$$p:\mathbb{R}\to V,\,p(x)=(|x|,|x|),$$
$$q:\mathbb{R}\to V,\,q(x)=(0,\Delta_{\mathbb{Q}}(x),$$
where $\Delta_{\mathbb{Q}}:\mathbb{R}\to\mathbb{R}$ acts by
$$\Delta_{\mathbb{Q}}(x)=\left\{\begin{array}{ll} 0, & \mbox{if }x\in\mathbb{Q},\\ 1, & \mbox{otherwise}\end{array}\right.\footnote{Obviously, $\Delta_{\mathbb{Q}}$ is related to the well-known Dirichlet function $D$ by $\Delta_{\mathbb{Q}}=1-D$. In fact, our choice to use $\Delta_{\mathbb{Q}}$ rather than the Dirichlet function turned out to be rather frivolous, since the same arguments as below, with very slight adjustments, work for both.}$$
\end{example}

We first observe that the space $V$ possesses the following (rather evident) property.

\begin{lemma}
The maximal isotropic subspace of $V$ coincides with $V$.
\end{lemma}

\begin{proof}
Let $f:V\to\mathbb{R}$ be a smooth function, and let $f(e_i)=a_i$ for $i=1,2$. Then $(f\circ p)(x)=(a_1+a_2)|x|$, which is an ordinary smooth map if and only if $a_2=-a_1$. Since $f\circ q=a_2\Delta_{\mathbb{Q}}$, we immediately obtain $a_1=a_2=0$.
\end{proof}

The property of $V$ that we are most interested in is the following one.

\begin{thm}
The space $V$ decomposes as a smooth direct sum of $\mbox{Span}(e_1)$ and $\mbox{Span}(e_2)$, with the subset diffeology on $\mbox{Span}(e_1)$ being generated by the map $x\mapsto(|x|,0)$ and that on $\mbox{Span}(e_2)$ by maps $x\mapsto(0,|x|)$ and $x\mapsto(0,\Delta_{\mathbb{Q}}(x)$.
\end{thm}

\begin{proof}
The essence of the proof consists in establishing the following equality:
$$|x|=2x\Delta_{\mathbb{Q}}(H_1(x))-2x\Delta_{\mathbb{Q}}(H_2(x))+x$$
for all $x\in\mathbb{R}$ and two specific functions $H_1,H_2:\mathbb{R}\to\mathbb{R}$ that we now define. The function $H_1$ is given by 
$$H_1(x)=\left\{\begin{array}{ll} e^{-\frac{1}{x^2}} & x>0 \\ 0 & x\leqslant 0.\end{array}\right.$$
To define $H_2$ we need certain auxiliary maps. Specifically, we define $w:(0,1)\to(\frac{1}{\sqrt2},1)$ by setting $w(x)=\frac{\sqrt2-1}{\sqrt2}x+\frac{1}{\sqrt2}$. Next, consider the countable sets $\{a_i\}=(0,1)\cap\mathbb{Q}$ and $\{b_j\}=w^{-1}\left((\frac{1}{\sqrt2},)\cap\mathbb{Q}\right)$; now, it is established in \cite{franklin} that there exists a smooth (actually analytic) monotone map $f:(0,1)\to(0,1)$ such that $f$ maps the set $\{a_i\}$ onto the set $\{b_j\}$ (that is, $f$ is a map $(0,1)\to(0,1)$ that sends rational points to rational points, and irrational ones, to irrational).

Let now $\gamma=w\circ f$, and define $\bar{\gamma}:[0,1)\to\mathbb{R}$ by $\bar{\gamma}(x)=\left\{\begin{array}{ll} \gamma(x) & x\in(0,1) \\ \frac{1}{\sqrt2} & x=0 \end{array}\right.$; observe that, since $\gamma$ is analytic at $0$, so is $\bar{\gamma}$. Finally, define 
$$H_2=\bar{\gamma}\circ H_1.$$
Observe that, since $f$ is analytic at $0$ (\cite{franklin}), $H_2$ is smooth.

The equality $|x|=2x\Delta_{\mathbb{Q}}(H_1(x))-2x\Delta_{\mathbb{Q}}(H_2(x))+x$ is now a simple consequence of the choice of the functions $H_1$ and $H_2$. Indeed, since $H_1(x)\in\mathbb{Q}$ if and only if $H_2(x)\in\mathbb{Q}$, we have $2x\Delta_{\mathbb{Q}}(H_1(x))-2x\Delta_{\mathbb{Q}}(H_2(x))=0$ for all $x>0$, while, since $H_1(x)\equiv 0$ and $H_2(x)\equiv\frac{1}{\sqrt2}$ for $x\leqslant 0$, it equals $-2x$ for $x\leqslant 0$. Hence the entire sum is $x$ for postive x and $-x$ for all other $x$, that is, is the absolute value function.
\end{proof}

\begin{rem}
As a mere curiosity, we observe that substituting in the expression of $|x|$ through $\delta_{\mathbb{Q}}$'s, in place of $x$ and $2x$, any other smooth functions $f$ and $g$, we obtain a map equaling $f$ on $(0,\infty)$ and $f-g$ on $(-\infty,0]$, therefore, utilizing if necessary appropriate translations, any piecewise-smooth function with one singularity can be expressed in an analogous way through $\Delta_{\mathbb{Q}}$.
\end{rem}

In addition, we can also conclude the following.

\begin{cor}
The subset diffeology on any non-zero subspace of $V$ is non-standard.
\end{cor}

\begin{proof}
Let $W$ be a (proper) subspace of $V$, and let $(a,b)$ be a  generator of it. It follows directly from the proof of the above theorem that the diffeology of $V$ contains plots of form $x\mapsto(|x|,0)$ and $x\mapsto(0,|x|)$; being a vector space diffeology, it therefore contains a plot of form $x\mapsto(a|x|,b|x|)$, which is a non-standard plot for the subset diffeology of~$V$.
\end{proof}

\section{Diffeological vector spaces not being coproducts}

We now turn to nonsmooth direct sums.

\paragraph{A nonsmooth direct sum decomposition} Let $V$ be $\mathbb{R}^3$ endowed with the vector space diffeology generated by the plot $p:\mathbb{R}\to V$ given by $p(x)=|x|(e_2+e_3)$. It was shown in \cite{me2018} that the subset diffeologies on its subspaces $V_0=\mbox{Span}(e_1,e_2)$ and $V_1=\mbox{Span}(e_3)$ are both standard, hence, while $V$ is an algebraic direct sum of $V_0$ and $V_1$, it is not their smooth direct sum.

In fact, for diffeological vector spaces non-smooth direct sum decompositions are a rather frequent phenomenon. To make this claim more precise, recall that, as was shown in \cite{me2018}, every finite-dimensional diffeological vector space $V$ contains at least one subspace that is maximal for the folowing two properties: its subset diffeology is standard, and it splits off as a smooth direct summand (however, contrary to what was erroneously claimed in \cite{me2018}, there is in general more than one such subspace). Any such subspace is called a \emph{characteristic subspace} of $V$, and it possesses the following property.

\begin{lemma}
Let $V_0$ be a characteristic subspace of $V$, and let $V=V_0\oplus V_1$ be a smooth direct sum decomposition of $V$. Then $V_1$ coincides with the maximal isotropic subspace of~$V$.
\end{lemma}

\begin{proof}
Let $f:V\to\mathbb{R}$ be a smooth function, and suppose that $f|_{V_1}\neq 0$. Since the decomposition $V=V_0\oplus V_1$ is smooth, the map $f'=0\oplus f|_{V_1}$ is a smooth linear function, which implies that the dimension of the diffeological dual of $V$ is strictly greater than that of $V_0$, which contradicts \cite{me2018}.
\end{proof}

Thus, there are numerous diffeological vector spaces admitting non-smooth decompositions.

\begin{cor}
If the dimension and the codimension of the maximal isotropic subspace of a finite-dimensional diffeological vector space $V$ are both positive then $V$ admits at least one smooth decomposition.
\end{cor}

\begin{proof}
As a (by assumption) proper subspace of $V$, any chosen characteristic subspace has a multitude of algebraic direct complements, only one of which coincides with the maximal isotropic subspace, all the others giving rise to non-smooth direct sum decompositions of~$V$.
\end{proof}

\paragraph{A non-complemented subspace} The discussion in the subsequent sections is contingent on the following assumption: \emph{suppose that there exists a non-smooth function $\gamma:\mathbb{R}\to\mathbb{R}$ such that, if $D_{\gamma}$ is the vector space diffeology on $\mathbb{R}$ generated by $\gamma$ and $D_{|\cdot|}$ is the vector space diffeology on $\mathbb{R}$ generated by the absolute value function, then $D_{\gamma}\cap D_{|\cdot|}$ is the standard diffeology on $\mathbb{R}$}.

Assuming this, we can easily conclude that there do exist finite-dimensional diffeological vector spaces containing non-complemented subspaces and, later, that there exist ones that do not admit any smooth direct sum decomposition.

\begin{example}
Let $\gamma$ be a function as in the assumption, and let $V$ be $\mathbb{R}^2$ endowed with the vector space diffeology generated by the following two plots: $p:x\mapsto(\gamma(x),\gamma(x))$ and $q:x\mapsto(0,|x|)$. Observe first that $V$ coincides with its maximal isotropic subspace; indeed, let $f:V\to\mathbb{R}$ defined by $f(e_1)=a$, $f(e_2)=b$, be a smooth linear function. Then the functions $x\mapsto(a+b)\gamma(x)$, $x\mapsto b|x|$ are ordinary smooth functions, which readily implies that $a=b=0$.

Observe now that the subspace $\mbox{Span}(e_1)$ of $V$ has standard subspace diffeology and therefore it follows from the above that it is not complemented. Indeed, a generic plot of $V$ locally has form $\left(\sum_{i=1}^kh_i\cdot(\gamma\circ H_i)+\alpha_1,\sum_{i=1}^kh_i\cdot(\gamma\circ H_i)+\sum_{j=1}^lf_j|F_j|+\alpha_2\right)$, where $\alpha_1,\alpha_2,h_i,H_i,f_j,F_j$ are some smooth functions $U\to\mathbb{R}$ for some domain $U\subseteq\mathbb{R}^n$. For this to be a plot of the subset diffeology of $\mbox{Span}(e_1)$ we must have $\sum_{i=1}^kh_i\cdot(\gamma\circ H_i)+\sum_{j=1}^lf_j|F_j|+\alpha_2\equiv 0$, that is, $\sum_{i=1}^kh_i\cdot(\Gamma\circ H_i)=-\sum_{j=1}^lf_j|F_j|-\alpha_2$. Now, the function on the right belongs to $D_{|\cdot|}$, while the one on the left belongs to $D_{\gamma}$, hence by assumption they are both smooth, which implies the desired conclusion.
\end{example}

Notice however that $V$, although (presumably) containing a non-complemented subspace, does admit an obvious decomposition into a smooth direct sum, that as $\mbox{Span}(e_1+e_2)\oplus\mbox{Span}(e_2)$.

Regarding the plausibility of the assumption stated on the beginning of the section and on which our conjectural example depends, we suggest that $\gamma$ could be a non-Baire (not belonging to any Baire class) function (for instance, if one assumes the axiom of choice, it could be the indicator function of a non-measirable set. Even without this axiom, there do exist Lebesgue-measurable functions that are not Borel-measurable --- again, indicator functions of sets with this property ---, which therefore again do not belong to any Baire class, see for instance \cite{brown}). It should on the other hand be noted that our assumption is not as trivial as one might perhaps at first glance exprct: for instance, it is known, for Baire 1 functions, that the class of a Baire function is not necessarily preserved by the operations of pre-composition with smooth functions, multiplication by such, and summation,\footnote{It is preserved by the latter two operations \cite{baireLin-comb}, but not necessarily by that of the pre-composition, there being a certain field of study of the relevant issue, see for instance \cite{baireKM}, \cite{baireZ}.} and we ourselves give in the previous section an example of of a Baire 2 function being rendered a Baire 0 function by precisely these operations. Yet, it can be noticed that, if it were true that any subspace of a finite-dimensional diffeological vector space is complemented, our example would imply that any function $\mathbb{R}\to\mathbb{R}$ whatsoever could be made continuous via these operations.

Still independently of such considerations, an attempt to construct an explicit example of a non-complemented subspace could be made as follows. Let $V$ be $\mathbb{R}^2$ endowed with the vector space diffeology generated by the plot $x\mapsto\left(\Delta_{\mathbb{Q}}(x),\Delta_{\mathbb{Q}}(\sqrt{|x|})\right)$. Observe first that the maximal isotropic subspace of $V$ coincidds with $V$. Indeed, let $f$, $f(e_1)=a$, $f(e_2)=b$, be a smooth linear function on $V$, then $x\mapsto a\Delta_{\mathbb{Q}}+b\Delta_{\mathbb{Q}}(\sqrt{|x|})$ is an ordinary smooth function $\mathbb{R}\to\mathbb{R}$. Since the set of irrational numbers is dense in $\mathbb{R}$, it must be a constant function with value $a+b$. On the other hand, since $\mathbb{R}$ contains both rational numbers with irrational root and those with rational root, this value must be equal to both $b$ and $0$, respectivly, hence $a=b=0$.

Observe next that the result just obtained implies that if $V$ contains a subspace with the subset diffeology that is standard, that subspace cannot split off as a smooth direct summand, and consider the subspace $\mbox{Span}(e_1)$. It is trivial to observe that this subspace having standard subset diffeology depends on the following, at the moment conjectural, implication holding for all smooth functions $h_i,H_i$:
$$\sum_{i=1}^kh_i(\Delta_{\mathbb{Q}}\circ \sqrt{|H_i|})\mbox{ is smooth }\Rightarrow\sum_{i=1}^kh_i(\Delta_{\mathbb{Q}}\circ H_i)\mbox{ is smooth }.$$
Whether this implication does actually hold is at the moment work in progress.\footnote{The hope of it being true is based on potential use of arguments akin to those involved in the study of recoverability of Baire fuctions (see for instance \cite{baireRecover}, \cite{baireRecover2}, although there regard Baire 1 functions): the potential recoverability of $\sum_{i=1}^kh_i(\Delta_{\mathbb{Q}}\circ \sqrt{|H_i|})$ from its values on certain specific sets and the assumption of it being smooth might impose certain conditions (local nullity conditions, in fact) on the coefficient functions $h_i$ sufficient to ensure the smoothness of $\sum_{i=1}^kh_i(\Delta_{\mathbb{Q}}\circ H_i)$.}

\paragraph{Non-decomposable vector spaces} Let $\gamma$ be a function as in the previous section, \emph{i.e.}, satisfying the assumption that $D_{\gamma}\cap D_{|\cdot|}$ is the standard diffeology, whose existence we again assume. Then this yields the following example of a diffeological vector space that does not admit any non-trivial decompositions into a smooth direct sum. 

\begin{example}
Let $W$ be $\mathbb{R}^2$ endowed with the vector space diffeology generated by the map $x\mapsto(|x|,\gamma(x))$. Observe first that the maximal isotropic subspace of $W$ coincides with $W$; indeed, if $f:W\to\mathbb{R}$ is a smooth linear map defined by $f(e_1)=a$, $f(e_2)=b$ then $x\mapsto a|x|$, $x\mapsto b\gamma(x)$ are ordinary smooth maps, which immediately implies $a=b=0$. 

We claim that every proper (nontrivial) subspace of $W$ has standard diffeology. Indeed, let $ce_1+de_2$ be a generator of such a subspace. A generic plot of W locally is a function of form $\left(\sum_{i=1}^kh_i|H_i|+\alpha_1,\sum_{i=1}^kh_i(\gamma\circ H_i)+\alpha_2\right)$ for some smooth maps $\alpha_1,\alpha_2,h_i,H_i:U\to\mathbb{R}$. For this to be a plot of the subset diffeology on $\mbox{Span}(ce_1+de_2)$ we must have $d\sum_{i=1}^kh_i|H_i|+d\alpha_1=c\sum_{i=1}^kh_i(\gamma\circ H_i)+c\alpha_2$. However, the function on the left belongs to $D_{|\cdot|}$, while that on the right, to $D_{\gamma}$, therefore by the assumption both of them are smooth. Thus, every one-dimensional subspace of $W$ has standard diffeology, and since the maximal isotropic subspace of $W$ is $W$ itself, no non-trivial subspace of $W$ with standard subbset diffeology splits off as a smooth direct summand. Therefore $W$ does not admit nontrivial decompositions into a smooth direct sum.
\end{example}

It is quite obvious that the choice of $|\cdot|$ is relatively arbitrary, and the same procedure as in the example above and that in the preceding section would work for any pair of non-smooth functions $\gamma_1,\gamma_2$ such that $D_{\gamma_1}\cap D_{\gamma_2}$ is the standard diffeology of $\mathbb{R}$. Since we are not aware of a formal proof of existence or non-existence of such functions, we again notice that a different kind of space, which in the previous section was conjectured to contain a non-complemented subspace, should this latter conjecture turn out to be correct, would turn out to be non-decomposable as well. Indeed, it is easy to see that the conjectural implication stated in the previous section would also imply that the subset diffeology on any proper subspace of the space in question be standard, and by the same reasoning the space would be nondecomposable.

\begin{rem}
As a side remark, the question of a given (finite-dimensional) diffeological space admitting at least one decomposition into a smooth direct sum of its subspaces can be also stated in the following terms. Let $V$ be the space in question, and let $\mathcal{D}_p$ be the category of its plots (\emph{i.e.} the category whose objects are plots of $V$ and whose arrows are commutative triangles of form 
\[
\begin{tikzcd}
U \arrow{rr}{f} \arrow[swap]{dr}{p} & & U' \arrow{dl}{q} \\ [5pt] & V
\end{tikzcd},
\]
where $p$ and $q$ are plots of $V$, and $f:U\to U'$ is an ordinary smooth map; this category is amply used in diffeology, for instance in studying the D-topology on diffeological spaces \cite{csw-top}, defining their tangent spaces and tangent bundles \cite{cw-tangent}, defining diffeology on the Milnor classifying space of a diffeological group \cite{mw}, defining sheaves for diffeological spaces \cite{krepskiMW}, and so on). 

Let $V_0$, $V_1$ be a pair of subspaces of $V$, and let $\mathcal{D}_p^0$ and $\mathcal{D}_p^1$ be respectively the categories of plots of the subset diffeologies on $V_0$ and $V_1$. Consider the subcategory $\mathcal{D}_p^{0,1}$ of the product category $\mathcal{D}_p^0\times\mathcal{D}_p^1$ whose objects are pairs $(p_0,p_1)$ such  that $p_0$ and $p_1$ have the same domain.\footnote{Obviously, $\mathcal{D}_p^{0,1}$ is just the equalizer in the category \textbf{Cat} of all small categories of the two functors $F_0,F_1:\mathcal{D}_p^0\times\mathcal{D}_p^1\to\mbox{\textbf{Cart}}$ into the category of all domains in Cartesian spaces, that are given by the compositions of the functor $\mathcal{D}_p^0\times\mathcal{D}_p^1\to\mbox{\textbf{Cart}}\times\mbox{\textbf{Cart}}$ that assigns to a pair of plots the pair of their domains, with the projections on factors.} Then it is obvious that $V$ decomposes as a smooth direct sum of $V_0$ and $V_1$ if and only if the assignment $(p_0,p_1)\mapsto p_0+p_1$ defines an isomorphism of categories $\mathcal{D}_p^{0,1}\to\mathcal{D}_p$. In particular, $V$ is decomposable if its category of plots $\mathcal{D}_p$ can be identified with some equalizer category $\mathcal{D}_p^{0,1}$.
\end{rem}

\paragraph{On decompositions of form $\mbox{Ker}(f)\oplus\mbox{Im}(f)$} Finally, we observe that if nondecomposable diffeological vector spaces do exist, for any such space a diffeomorphism $V\cong\mbox{Ker}(f)\oplus\mbox{Im}(f)$ ($\mbox{Im}(f)$ being considered, recall, with the pushforward diffeology) automatically cannot hold for any nontrivial $f$ with nontrivial kernel (this by nature of what it means to be a nondecomposable space). On the other hand, the subspace $\mbox{Ker}(f)$ being a summand in some non-smooth decomposition does not, of course, exclude that there may be such a diffeomorphism, as can be illustrated by the space $V$ of \cite{me2018} ($\mathbb{R}^3$ with the vector space diffeology generated by $x\mapsto|x|(e_2+e_3)$) and the map $f$ defined on $V$ by $e_1\mapsto e_1$, $e_2\mapsto e_2$, $e_3\mapsto 0$, and taking values in $\mathbb{R}^2$ endowed with pushforward diffeology.

\vspace{1cm}

\noindent University of Pisa \\
Department of Mathematics \\
Via F. Buonarroti 1C\\
56127 PISA -- Italy\\
\ \\
ekaterina.pervova@unipi.it\\

\end{document}